\documentclass[11pt]{article}

\usepackage[a4paper,margin=1in]{geometry}
\usepackage{amsmath,amssymb,amsthm,mathtools}
\usepackage{enumitem}
\usepackage{graphicx}
\usepackage[colorlinks=true,linkcolor=red,citecolor=red,urlcolor=red]{hyperref}

\newtheorem{theorem}{Theorem}[section]
\newtheorem{proposition}[theorem]{Proposition}
\newtheorem{lemma}[theorem]{Lemma}

\theoremstyle{remark}

\newcommand{\R}{\mathbb{R}}

\title{\bf Isolated singularities of weak solutions to the Serrin-critical \(p\)-Laplace equation}
\author{Meiqing Xu \and Hui Yang}
\date{\today}

\begin{document}
\maketitle

\begin{abstract}
We classify isolated singularities of nonnegative weak solutions to the Serrin-critical quasilinear Lane-Emden equation
\[
   -\Delta_pu=u^\frac{n(p-1)}{n-p}\qquad \text{in }B_1\setminus\{0\},
\]
where $1<p<n$, $ \Delta_p$ is the classical $p$-Laplacian operator and $B_1\setminus\{0\} \subset \mathbb{R}^n$ is the punctured unit ball. We prove that either the origin is a removable singularity or $u$ satisfies
\[
 \lim_{x\to0}
 |x|^{\frac{n-p}{p-1}}
 \left(\log\frac1{|x|}\right)^{\frac{n-p}{p(p-1)}}u(x)
 =
 \left[
 \frac{n-p}{p}
 \left(\frac{n-p}{p-1}\right)^{p-1}
 \right]^{\frac{n-p}{p(p-1)}}.
\]
This extends the result of Aviles (Comm. Math. Phys. 1987: 177--192) from the semilinear case $p=2$ to the full range $1<p<n$, without assuming radial symmetry.
We develop a new BV-based argument
for weak solutions. We also give a direct construction of positive radial singular solutions with the same asymptotic behavior.
\end{abstract}

\section{Introduction}
\label{sec:introduction}

The classical $p$-Laplace operator
\[
 \Delta_pu:=\operatorname{div}\bigl(|\nabla u|^{p-2}\nabla u\bigr),
 \qquad p>1,
\]
is a basic model for nonlinear media and arises in a wide range of
problems involving nonlinear diffusion, non-Newtonian fluids, and
nonlinear elasticity
\cite{Ball1977,DiBenedettoHerrero1990,MalekRajagopalRuzicka1995}.
We consider the positive solutions of the quasilinear Lane--Emden equation
\[
 -\Delta_pu=u^s,
 \qquad u\ge0,\quad 1<p<n,\quad s>p-1.
\]
For $1<p<n$, denote the Serrin and Sobolev exponents by
\[
 q_{\mathrm{Ser}}:=\frac{n(p-1)}{n-p}
 \quad \text{and} \quad
 q_{\mathrm{Sob}}:=\frac{np}{n-p}-1,
\]
respectively.

When $p=2$ and $n\ge3$, the equation reduces to $-\Delta u=u^s$. In the Sobolev subcritical and critical ranges, the classification of positive entire solutions has been completely established; see \cite{CaffarelliGidasSpruck1989,ChenLi1991,GidasNiNirenberg1979,GidasSpruck1981,LiZhang2003}. Since the equation is conformally invariant in the critical case, the moving-plane and moving-sphere methods play a central role \cite{CaffarelliGidasSpruck1989,ChenLi1991,LiZhang2003,LiZhu1995}. For isolated singularities, Lions \cite{Lions1980} studied the range
$1<s<n/(n-2)$, below the Serrin exponent, while
Gidas--Spruck \cite{GidasSpruck1981} studied the range
$n/(n-2)<s<(n+2)/(n-2)$, between the Serrin and Sobolev exponents.
In the first range, a positive solution in a punctured ball either
extends across the origin or satisfies
\[
 u(x)\sim c|x|^{2-n}
 \qquad\text{as }x\to0
\]
for some $c>0$ depending on the solution. In the second range, every nonremovable positive solution satisfies
\[
 u(x)\sim C_{n,s}|x|^{-\frac{2}{s-1}}
 \qquad\text{as }x\to0.
\]
At the Serrin exponent $s=n/(n-2)$, the power $|x|^{-2/(s-1)}$ becomes $|x|^{2-n}$, which is harmonic away from the origin. At this endpoint, nonremovable solutions are no longer asymptotic to a positive multiple of $|x|^{2-n}$. Aviles \cite{Aviles1983,Aviles1987} proved the exact limit
\[
 \lim_{x\to0}|x|^{n-2}
 \left(\log\frac1{|x|}\right)^{\frac{n-2}{2}}u(x)
 =\left(\frac{(n-2)^2}{2}\right)^{\frac{n-2}{2}}
\]
for every nonremovable nonnegative solution.
%
At the Sobolev exponent $s=(n+2)/(n-2)$, asymptotic radial symmetry
and refined asymptotics near isolated singularities were established
in \cite{CaffarelliGidasSpruck1989,HanLiLiCPAM,KorevaarMazzeoPacardSchoen1999,Li1996}, while isolated singularities
for Serrin-critical differential inequalities were studied by Taliaferro
\cite{Taliaferro2001}.

When $p\ne2$, Serrin--Zou \cite{SerrinZou2002} proved
nonexistence of positive entire solutions for
$p-1<s<q_{\mathrm{Sob}}$. At the Sobolev critical exponent $s=q_{\mathrm{Sob}}$, the
classification of positive finite-energy solutions was established
through the works \cite{DamascelliMerchanMontoroSciunzi2014,Sciunzi2016,Vetois2016}. Further classification results without global assumptions were recently obtained in  \cite{CatinoMonticelliRoncoroni2023,Ou2025,Vetois2024,Zhang2026}.
The study of isolated singularities for quasilinear elliptic equations
goes back to the fundamental work of Serrin
\cite{Serrin1964,Serrin1965}. Kichenassamy--V\'eron
\cite{KichenassamyVeron1986} established precise results for
singular $p$-harmonic functions, while Bidaut-V\'eron
\cite{BidautVeron1989} developed local estimates, integrability,
and removability results for quasilinear equations of
Emden--Fowler type. Nonlinear potential estimates of
Kilpel\"ainen--Mal\'y \cite{KilpelainenMaly1994} provide another
important tool for this analysis.

For the equation $-\Delta_pu=u^s$, Guedda--V\'eron
\cite{GueddaVeron1988} classified the isolated singularities of positive {\it radial} solutions throughout $p-1<s<q_{\mathrm{Sob}}$.
Their results show that a radial solution either extends across the
origin or has one of the following asymptotic behaviors:
\[
 u(r)\sim
 \begin{cases}
 c\,r^{-\frac{n-p}{p-1}},
   & p-1<s<q_{\mathrm{Ser}},\\[2mm]
 A_{n,p}\,r^{-\frac{n-p}{p-1}}
   \left(\log\dfrac1r\right)^{-\frac{n-p}{p(p-1)}},
   & s=q_{\mathrm{Ser}},\\[2mm]
 C_{n,p,s}\,r^{-\frac{p}{s-p+1}},
   & q_{\mathrm{Ser}}<s<q_{\mathrm{Sob}},
 \end{cases}
 \qquad \text{as } r \to 0.
\]
Here $c>0$ depends on the solution, $A_{n,p}$ is the constant
in \eqref{eq:sharpconstant}, and
\[
 C_{n,p,s}:=
 \left[
 \left(\frac{p}{s-p+1}\right)^{p-1}
 \left(n-p-\frac{p(p-1)}{s-p+1}\right)
 \right]^{\frac1{s-p+1}}.
\]
In particular, both the logarithmic exponent and its exact coefficient were already known for radial solutions in the Serrin-critical case. In the range $p-1<s<q_{\mathrm{Ser}}$, Bidaut-V\'eron \cite{BidautVeron1989} obtained the same
fundamental-solution asymptotics without assuming radial symmetry.
Du--Zhou \cite{DuZhou2024} extended the radial analysis to the weighted equation $-\Delta_pu=|x|^\alpha u^s$ and proved a corresponding nonradial result in the weighted subcritical Serrin regime. More directly related to the present problem, Ma--Zang
\cite{MaZang2025} considered positive $C^2$ solutions of the quasilinear differential inequality
\[
 0\le-\Delta_pu\le\tau u^{q_{\mathrm{Ser}}},
 \qquad \tau>0,
\]
without assuming radial symmetry. They proved asymptotic spherical
symmetry and classified the singular behavior into three cases:
removability, asymptotic equivalence to a positive multiple of the
fundamental solution, or a nonremovable singularity of smaller order.
In the last case, their result gives
\[
 \lim_{x\to0}|x|^{\frac{n-p}{p-1}}u(x)=0,
 \qquad
 \liminf_{x\to0}|x|^{\frac{n-p}{p-1}}
 \left(\log\frac1{|x|}\right)^{\frac{n-p}{p(p-1)}}u(x)
 \ge \tau^{-\frac{n-p}{p(p-1)}}A_{n,p}.
\]
Thus their theorem determines a sharp logarithmic lower bound, but
does not assert an exact logarithmic limit for the differential
inequality.

For the equation $-\Delta_pu=u^{q_{\mathrm{Ser}}}$ with a singularity at the origin, we shall prove that every nonnegative singular weak solution satisfies
\[
 \lim_{x\to0}|x|^{\frac{n-p}{p-1}}
 \left(\log\frac1{|x|}\right)^{\frac{n-p}{p(p-1)}}u(x)=A_{n,p}.
\]
Compared with the radial results
of Guedda--V\'eron \cite{GueddaVeron1988}, our conclusion applies to solutions without any prescribed symmetry. Moreover, compared with Ma--Zang \cite{MaZang2025}, we obtain the matching
sharp upper bound for the equation and hence the full limit, while working in the class of weak solutions.
For convenience, throughout the paper we set
\begin{equation}\label{eq:parameters}
 q:=q_{\mathrm{Ser}} = \frac{n(p-1)}{n-p}.
\end{equation}
Define
\begin{equation}\label{eq:sharpconstant}
 A_{n,p}
 :=
 \left[
 \frac{n-p}{p}
 \left(\frac{n-p}{p-1}\right)^{p-1}
 \right]^{\frac{n-p}{p(p-1)}}.
\end{equation}
For $1<p<n$, a nonnegative function
$u\in W^{1,p}_{\mathrm{loc}}(B_1\setminus\{0\})$ is called a \emph{weak solution} of
\begin{equation}\label{eq:main}
 -\Delta_pu=u^q
 \qquad\text{in }B_1\setminus\{0\},
\end{equation}
if
\[
 \int_{B_1\setminus\{0\}}
 |\nabla u|^{p-2}\nabla u\cdot\nabla\varphi\,dx
 =
 \int_{B_1\setminus\{0\}}u^q\varphi\,dx
\]
for every $\varphi\in C_c^\infty(B_1\setminus\{0\})$.
Density extends the identity from
$C_c^\infty(B_1\setminus\{0\})$ to every test function in $W^{1,p}_0(D)$ for every $D\Subset B_1\setminus\{0\}$.

\begin{theorem}[Classification of singularities in the Serrin-critical case]\label{thm:main}
Let $1<p<n$ and let $u\in W^{1,p}_{\mathrm{loc}}(B_1\setminus\{0\})$ be a nontrivial nonnegative weak solution of \eqref{eq:main}.
Then precisely one of the following alternatives holds:
\begin{enumerate}[label=\textup{(\roman*)}]
 \item $0$ is removable: $u$ extends to a weak solution in $B_1$, and the
       extension belongs to $C^{1,\alpha}_{\mathrm{loc}}(B_1)$ for some
       $\alpha\in(0,1)$;
 \item $0$ is nonremovable and
 \begin{equation}\label{eq:final-limit}
  \lim_{x\to0}
 |x|^{\frac{n-p}{p-1}}
 \left(\log\frac1{|x|}\right)^{\frac{n-p}{p(p-1)}}u(x)
  =A_{n,p}.
 \end{equation}
\end{enumerate}
\end{theorem}

We complement this classification with a direct radial construction,
independent of Theorem~\ref{thm:main}, which shows that the nonremovable
alternative is attained.

\begin{theorem}[Existence of logarithmic singularities]
\label{prop:radial-existence}
Let $1<p<n$, and let $q$ and $A_{n,p}$ be given by
\eqref{eq:parameters} and \eqref{eq:sharpconstant}, respectively. Then there exists a positive radial solution $u(x)=U(|x|)\in C^\infty(B_1\setminus\{0\})$
of \eqref{eq:main} satisfying \eqref{eq:final-limit}.
\end{theorem}

The proof of Theorem~\ref{thm:main} uses a different mechanism from the semilinear argument of Aviles \cite{Aviles1983,Aviles1987}, which combines spherical averaging with energy estimates on the cylinder. Since spherical averaging does not commute with the $p$-Laplacian for $p\ne2$, this reduction is not directly available.
Instead, we first prove that the ratio of the spherical maximum to the spherical minimum tends to $1$ as \(|x| \to 0\). Then we establish an asymptotic upper bound for the minimum and an asymptotic lower bound for the maximum, with the same constant. The convergence of their ratio then forces both normalized quantities to converge to this constant, yielding the exact limit in \eqref{eq:final-limit}.

In contrast to the $C^2$ framework of Ma--Zang
\cite{MaZang2025}, we develop a new argument based on bounded variation (BV), which extends
the spherical-extrema method to continuous weak solutions.
We derive one-dimensional differential inequalities in the sense
of Radon measures and use the BV chain rule to account for the
singular parts of the distributional derivatives. This allows us to work directly with the spherical minimum and
maximum, which are locally Lipschitz but need not be twice
differentiable, and to establish the exact logarithmic asymptotics
throughout $1<p<n$.

The paper is organized as follows. Section~\ref{sec:preliminaries} collects the local regularity results and preliminary estimates.
Section~\ref{sec:sharp-asymptotics} proves Theorem~\ref{thm:main}. Section~\ref{sec:radial-existence} proves Theorem~\ref{prop:radial-existence} by constructing a positive radial solution with the exact logarithmic asymptotics.

\section{Preliminaries}
\label{sec:preliminaries}

We first record that $u$ is locally $C^{1,\alpha}$, strictly positive, and a viscosity solution away from the origin.

\begin{lemma}\label{lem:local-regularity}
Every weak solution in Theorem~\ref{thm:main} satisfies
\(
 u\in C^{1,\alpha}_{\mathrm{loc}}(B_1\setminus\{0\})
\)
for some $\alpha\in(0,1)$. Moreover, $u>0$ in
$B_1\setminus\{0\}$ and is a
viscosity solution of \eqref{eq:main}.
\end{lemma}

\begin{proof}
By the Sobolev embedding, the subcriticality assumption, and Serrin's local boundedness theorem \cite{Serrin1964}, we obtain
$u\in L^\infty_{\mathrm{loc}}(B_1\setminus\{0\})$. Standard interior regularity estimates \cite{DiBenedetto1983,Tolksdorf1984} then yield
$u\in C^{1,\alpha}_{\mathrm{loc}}(B_1\setminus\{0\})$ for some $\alpha\in(0,1)$.
Set \(g(x):=u(x)^q\). Then \(g\) is continuous and \(u\) is a weak solution of
\(-\Delta_p u=g(x)\). By \cite[Theorem~1.8]{MedinaOchoa2019}, \(u\) is also a viscosity solution of \(-\Delta_pu=u^q\).
Finally, the strong maximum principle \cite[Theorem~5]{Vazquez1984}
gives $u>0$, because $u\ge0$ and $u\not\equiv0$.
\end{proof}

The following proposition gives the integrability and the uniform bound of $u$ near the origin.

\begin{proposition}
\label{prop:potential-input}
Let $u$ be as in Theorem~\ref{thm:main}.  Then
\begin{equation}\label{eq:uq-integrable}
 u^q\in L^1(B_{1/2}),
\end{equation}
and there is $C>0$ such that
\begin{equation}\label{eq:fundamental-upper}
 u(x)\le C|x|^{-\frac{n-p}{p-1}},
 \qquad 0<|x|<\frac14.
\end{equation}
\end{proposition}

\begin{proof}
The continuous weak supersolution $u$ is $p$-superharmonic.  The
isolated-point extension theorem \cite[Theorem~1.1]{BidautVeron1989} gives a nonnegative Radon measure on a
full ball of the form
\[
 -\Delta_pu=g+\lambda\delta_0,
 \qquad g\in L^1_{\mathrm{loc}},\quad \lambda\ge0.
\]
In the punctured ball, the weak equation identifies $g=u^q$ almost
everywhere, which proves \eqref{eq:uq-integrable}.
The estimate \eqref{eq:fundamental-upper} follows from \cite[Theorem IV (b)]{SerrinZou2002}. We may also refer to \cite[Lemmas~3.1--3.3]{MaZang2025}.
\end{proof}

The next lemma shows that $u$ is uniformly of smaller order than the fundamental solution near the origin.

\begin{lemma}[Uniformly weaker than the fundamental solution]
\label{lem:zero-mass-uniform}
One has
\begin{equation}\label{eq:uniform-zero-mass}
 \lim_{x\to0}|x|^{\frac{n-p}{p-1}}u(x)=0.
\end{equation}
\end{lemma}

\begin{proof}
Write
\[
 b(x):=u(x)^{q-p+1}.
\]
Then \eqref{eq:main} is
\[
 -\Delta_pu=b(x)u^{p-1}.
\]
By \eqref{eq:fundamental-upper},
\[
|x_0|^p
\|b\|_{L^\infty(B(x_0,|x_0|/4))}
\le M
\]
uniformly for sufficiently small $|x_0|$. Harnack inequality \cite[Theorem~7]{Serrin1964}, applied after
rescaling $B(x_0,|x_0|/4)$ to the unit ball, yields
\begin{equation}\label{eq:local-harnack}
 \inf_{B(x_0,|x_0|/8)}u\ge c\,u(x_0),
\end{equation}
where $c>0$ is independent of $x_0$.

If \eqref{eq:uniform-zero-mass} failed, there would be $\varepsilon>0$
and $x_j\to0$ such that
\[
 |x_j|^{\frac{n-p}{p-1}}u(x_j)\ge\varepsilon.
\]
Passing to a subsequence, assume $|x_{j+1}|\le|x_j|/4$, so that the balls
$B_j:=B(x_j,|x_j|/8)$ are pairwise disjoint.  By \eqref{eq:local-harnack},
\[
 u(x)\ge c\varepsilon |x_j|^{-\frac{n-p}{p-1}}
 \qquad\text{for }x\in B_j.
\]
Since $\frac{n-p}{p-1}q=n$,
\[
 \int_{B_j}u^q\,dx
 \ge c_0\varepsilon^q|x_j|^{n-\frac{n-p}{p-1}q}
 =c_0\varepsilon^q.
\]
Summing over the disjoint balls contradicts
$u^q\in L^1(B_{1/2})$.  Hence \eqref{eq:uniform-zero-mass} holds.
\end{proof}


The following lemma shows that either the singularity at $0$
is removable or $u(x)\to+\infty$ as $x\to0$.

\begin{lemma}\label{lem:dichotomy}
Either $0$ is removable, or
\begin{equation}\label{eq:divergence}
 u(x) \to +\infty
 \qquad\text{as }x\to0.
\end{equation}
\end{lemma}

\begin{proof}
Fix \(0<\rho<1\), and introduce the fixed annuli
\[
A_\rho:=\{y\in\mathbb R^n:\rho\le |y|\le \rho^{-1}\},
\qquad
\widetilde A_\rho
:=\left\{y\in\mathbb R^n:\frac{\rho}{2}\le |y|
\le 2\rho^{-1}\right\}.
\]
For \(r>0\) sufficiently small, define
\[
u_r(y):=u(ry),
\qquad y\in\widetilde A_\rho.
\]
By scaling the equation, \(u_r\) satisfies
\[
-\Delta_p u_r
=
r^p u_r^q
=
a_r(y)u_r^{p-1}
\quad\text{in }\widetilde A_\rho,
\]
where
\[
a_r(y):=r^p u(ry)^{q-p+1}=
|y|^{-p}
\left(|ry|^{\frac{n-p}{p-1}}u(ry)\right)^{q-p+1}.
\]
By \eqref{eq:uniform-zero-mass},
\[
a_r(y) \to 0
\qquad\text{as }r\to0
\]
uniformly for \(y\in\widetilde A_\rho\).
In particular, after decreasing \(r\) if necessary,
\[
\|a_r\|_{L^\infty(\widetilde A_\rho)}\le1.
\]
Harnack inequality applied after rescaling gives a constant
\(C_\rho>1\), independent of \(r\), such that
\[
\sup_{A_\rho}u_r
\le
C_\rho\inf_{A_\rho}u_r.
\]
This gives for every fixed $0<\rho<1$, a
constant $C_\rho$, independent of \(r\), such that
\begin{equation}\label{eq:annular-harnack}
 \sup_{\rho r\le |x|\le \rho^{-1}r}u(x)
 \le C_\rho
 \inf_{\rho r\le |x|\le \rho^{-1}r}u(x)
\end{equation}
for all sufficiently small $r$.

Suppose that $u$ does not tend to infinity.  If $u$ were nevertheless
unbounded near zero, \eqref{eq:annular-harnack} would allow us to choose
three radii
\[
 r_j^-<s_j<r_j^+\longrightarrow0
\]
such that
\[
 \min_{|x|=r_j^-}u(x)\to\infty,
 \qquad
 \max_{|x|=s_j}u(x)\le C,
 \qquad
 \min_{|x|=r_j^+}u(x)\to\infty.
\]
Since $u$ is $p$-superharmonic, the maximum principle with the constant
\[
 L_j:=\min\left\{
 \min_{|x|=r_j^-}u,
 \min_{|x|=r_j^+}u
 \right\}
\]
in the annulus $B_{r_j^+}\setminus\overline{B}_{r_j^-}$ gives
$u\ge L_j$ throughout $B_{r_j^+}\setminus\overline{B}_{r_j^-}$, contradicting the bounded value on
$|x|=s_j$.  Thus $u$ is bounded near zero.


In this case, we show that
the singularity at \(0\) is removable. Choose \(R>0\) sufficiently small and let
\[
M:=\|u\|_{L^\infty(B_{2R})}.
\]
Choose \(\chi\in C^\infty([0,\infty))\), \(0\le\chi\le1\), and
\[
\chi(r)=0\quad\text{for }0\le r\le1,
\qquad
\chi(r)=1\quad\text{for }r\ge2.
\]
For \(\varepsilon>0\), define $\chi_\varepsilon(x):=
\chi\!\left(\frac{|x|}{\varepsilon}\right)$.
Then
\[
\begin{aligned}
\int_{\mathbb R^n}|\nabla\chi_\varepsilon|^p\,dx \le C\,\varepsilon^{\,n-p}\to 0
\qquad\text{as }\varepsilon\downarrow0.
\end{aligned}
\]
Let \(\eta\in C_c^\infty(B_{2R})\) satisfy
\(0\le\eta\le1\) and \(\eta\equiv1\) in \(B_R\), and set
\[
\psi_\varepsilon:=\eta\chi_\varepsilon.
\]
Since \(\psi_\varepsilon\) vanishes near \(0\), the function
\(\psi_\varepsilon^p u\) is an admissible test function in the punctured ball.
Testing \eqref{eq:main} with \(\psi_\varepsilon^p u\), we obtain
\begin{align*}
 \int_{B_{2R}}\psi_\varepsilon^p|\nabla u|^p\,dx
 =\int_{B_{2R}}u^{q+1}\psi_\varepsilon^p\,dx-p\int_{B_{2R}}u\psi_\varepsilon^{p-1}|\nabla u|^{p-2}
          \nabla u\cdot\nabla\psi_\varepsilon\,dx.
\end{align*}
The last term is estimated by
\begin{align*}
 p u\psi_\varepsilon^{p-1}|\nabla u|^{p-1}|\nabla\psi_\varepsilon|
 &\le\frac12\psi_\varepsilon^p|\nabla u|^p
    +C_pu^p|\nabla\psi_\varepsilon|^p.
\end{align*}
Thus
\[
\int_{B_{2R}}\psi_\varepsilon^p|\nabla u|^p\,dx
\le
C\int_{B_{2R}}u^p|\nabla\psi_\varepsilon|^p\,dx
+
C\int_{B_{2R}}u^{q+1}\psi_\varepsilon^p\,dx.
\]
Letting
\(\varepsilon\downarrow0\) yields
\[
\int_{B_R}|\nabla u|^p\,dx<\infty.
\]
Thus the bounded extension of \(u\) across the single point \(0\) belongs to
\(W^{1,p}(B_R)\).

It remains to verify the equation across \(0\). Let
\(\varphi\in C_c^\infty(B_R)\). Since
\(\varphi\chi_\varepsilon\) vanishes near \(0\), it is an admissible test
function, and hence
\[
\int \chi_\varepsilon|\nabla u|^{p-2}\nabla u\cdot\nabla\varphi\,dx
+\int \varphi|\nabla u|^{p-2}\nabla u\cdot\nabla\chi_\varepsilon\,dx
=
\int u^q\varphi\chi_\varepsilon\,dx.
\]
The term involving \(\nabla\chi_\varepsilon\) tends to zero, since
\[
\left|
\int
|\nabla u|^{p-2}\nabla u\cdot
\varphi\nabla\chi_\varepsilon\,dx
\right|
\le
\|\varphi\|_\infty
\|\nabla u\|_{L^p(B_R)}^{p-1}
\|\nabla\chi_\varepsilon\|_{L^p}
\to 0.
\]
The remaining terms converge by dominated convergence, and therefore
\[
\int_{B_R}|\nabla u|^{p-2}\nabla u\cdot\nabla\varphi\,dx
=
\int_{B_R}u^q\varphi\,dx.
\]
Hence, \(u\) extends as a weak solution of \eqref{eq:main} across \(0\). Since \(u^q\in L^\infty(B_R)\), the standard local regularity theory yields
\[
u\in C^{1,\alpha}_{\mathrm{loc}}(B_R)
\]
for some \(\alpha\in(0,1)\). Thus, the singularity is removable.
\end{proof}

\section{Proof of Theorem~\ref{thm:main}}
\label{sec:sharp-asymptotics}

In this section, we prove Theorem~\ref{thm:main}.
By Lemma~\ref{lem:dichotomy}, it remains to consider the case
of a nonremovable singularity. Throughout this section, we therefore
assume that $0$ is nonremovable, so that
\eqref{eq:uniform-zero-mass} and \eqref{eq:divergence} hold. For $0<r<1$, inspired by Ma--Zang \cite{MaZang2025} we define
\begin{equation*}
 \underline u(r):=\min_{|x|=r}u(x)
 \quad \text{and} \quad
 \overline u(r):=\max_{|x|=r}u(x).
\end{equation*}

We first prove that the spherical minimum and maximum are asymptotically equivalent near the nonremovable singularity.

\begin{lemma}[Asymptotic spherical symmetry]
\label{lem:asymptotic-symmetry}
In the nonremovable case,
\begin{equation}\label{eq:ratio-one}
 \frac{\overline u(r)}{\underline u(r)} \to 1
 \qquad\text{as }r\to0.
\end{equation}
\end{lemma}

\begin{proof}
Let $r_j\downarrow0$ be arbitrary and put
\[
 m_j:=\underline u(r_j),
 \qquad
 U_j(y):=\frac{u(r_jy)}{m_j}.
\]
By \eqref{eq:annular-harnack}, on every compact annulus
$K\Subset\R^n\setminus\{0\}$ there is $C_K>1$ such that
\begin{equation}\label{eq:29}
 C_K^{-1}\le U_j\le C_K
\end{equation}
for all large $j$.
Indeed, choose \(0<\rho<1\) such that
\[
K\cup \partial B_1\subset\{y:\rho\le |y|\le \rho^{-1}\}.
\]
\eqref{eq:annular-harnack} gives
\[  \sup_{\rho r_j\le |x|\le \rho^{-1}r_j}u(x)\le C_\rho \inf_{\rho r_j\le |x|\le \rho^{-1}r_j}u(x).  \]
 Since
\[
\inf_{\rho r_j\le |x|\le \rho^{-1}r_j}u(x)\le m_j\le \sup_{\rho r_j\le |x|\le \rho^{-1}r_j}u(x),
\]
we obtain
\[
C_\rho^{-1}m_j\le u(r_jy)\le C_\rho m_j
\qquad (y\in K).
\]
Dividing by \(m_j\) yields \eqref{eq:29}.
Moreover,
\[
 -\Delta_pU_j=\lambda_jU_j^q,
 \qquad
 \lambda_j:=r_j^pm_j^{q-p+1}
 =\bigl(r_j^{\frac{n-p}{p-1}}m_j\bigr)^{q-p+1}\to 0
\]
by \eqref{eq:uniform-zero-mass}.
For every compact annulus
\(K\Subset\mathbb R^n\setminus\{0\}\), choose
\(K\Subset K'\Subset\mathbb R^n\setminus\{0\}\).
The preceding Harnack estimate gives
\[
\|U_j\|_{L^\infty(K')}\le C_{K'},\qquad
\|\lambda_jU_j^q\|_{L^\infty(K')}
\le \lambda_j C_{K'}^q\le C
\]
for all large \(j\). Hence, the standard local regularity theory for the
nonhomogeneous \(p\)-Laplace equation yields
\[
\|U_j\|_{C^{1,\alpha}(K)}\le C_K,
\]
with \(C_K\) independent of \(j\).
After passing to a subsequence,
\[
 U_j\to  U
 \qquad\text{in }C^1_{\mathrm{loc}}(\R^n\setminus\{0\}),
\]
where $U>0$ is $p$-harmonic in $\R^n\setminus\{0\}$.
Every positive $p$-harmonic function in the punctured space has the form
\begin{equation*}
 U(y)=A|y|^{-\frac{n-p}{p-1}}+B,
 \qquad A,B\ge0;
\end{equation*}
see \cite[Theorem~2.2]{KichenassamyVeron1986} or
\cite[Theorem~5.3]{MaZang2025}.  Since
\[
 \min_{|y|=1}U_j(y)=1,
\]
we have $A+B=1$.  In particular, $U$ is identically equal to $1$ on
$\partial B_1$, and therefore
\[
 \frac{\overline u(r_j)}{\underline u(r_j)}
 =\max_{|y|=1}U_j(y) \to 1.
\]
The original sequence $r_j$ was arbitrary, so \eqref{eq:ratio-one}
holds.
\end{proof}


The spherical minimum and maximum are, respectively, a weak supersolution and a weak subsolution of \eqref{eq:main}.

\begin{lemma}\label{lem:envelope-inequalities}
The radial functions $\underline u(|x|)$ and
$\overline u(|x|)$ belong to
$W^{1,\infty}_{\mathrm{loc}}(B_1\setminus\{0\})$ and satisfy
\begin{align}
 -\Delta_p\underline u&\ge \underline u^q,
 \label{eq:lower-super}\\
 -\Delta_p\overline u&\le \overline u^q
 \label{eq:upper-sub}
\end{align}
weakly in $B_1\setminus\{0\}$.
\end{lemma}

\begin{proof}
For any orthogonal transformation $Q\in O(n)$, set $u_Q(x):=u(Qx)$.  Rotational invariance gives
$-\Delta_pu_Q=u_Q^q$, and
\[
 \underline u(|x|)=\min_{Q\in O(n)}u_Q(x),
 \qquad
 \overline u(|x|)=\max_{Q\in O(n)}u_Q(x).
\]
Let
\[
K=\{x\in\mathbb R^n:a\le |x|\le b\}
\Subset B_1\setminus\{0\}
\]
be a compact annulus. For all \(x,y\in K\),
\[
|u_Q(x)-u_Q(y)|
\le L_K|x-y|,
\]
uniformly with respect to \(Q\). Therefore,
\[
\begin{aligned}
|\underline u(|x|)-\underline u(|y|)|&\le L_K|x-y|,\\
|\overline u(|x|)-\overline u(|y|)|&\le L_K|x-y|,
\qquad x,y\in K,
\end{aligned}
\]
and thus both radial functions belong to
\(W^{1,\infty}_{\mathrm{loc}}(B_1\setminus\{0\})\).

By Lemma~\ref{lem:local-regularity}, each \(u_Q\) is a viscosity solution of
\(-\Delta_p u_Q=u_Q^q\).
Fix \(x_0\in B_1\setminus\{0\}\), and choose an orthogonal
transformation \(Q_0\) such that
\[
\underline u(|x_0|)=u_{Q_0}(x_0).
\]
Let \(\phi\) be a \(C^2\) function touching \(\underline u\) from below at
\(x_0\).
Since
\[
\phi(x)\le\underline u(|x|)\le u_{Q_0}(x)
\quad\text{near }x_0,
\qquad
\phi(x_0)=\underline u(|x_0|)=u_{Q_0}(x_0),
\]
the same function \(\phi\) touches \(u_{Q_0}\) from below at \(x_0\).
The viscosity supersolution property gives
\[
-\Delta_p\phi(x_0)
\ge u_{Q_0}(x_0)^q
=\underline u(|x_0|)^q.
\]
Thus, \eqref{eq:lower-super} holds in the viscosity sense. The analogous
upper-contact argument gives \eqref{eq:upper-sub} in the viscosity sense.
%
%
Applying the nonhomogeneous weak--viscosity equivalence \cite[Theorem~3.2 and Theorem~4.1]{JulinJuutinen2012}, locally on compact annuli, gives \eqref{eq:lower-super} and \eqref{eq:upper-sub}.
\end{proof}


Put
\begin{equation}\label{eq:EF}
 t:=\frac{n-p}{p-1}\log\frac{n-p}{(p-1)r},
 \qquad
 r=\frac{n-p}{p-1}e^{-\frac{p-1}{n-p}t},
\end{equation}
and define
\begin{equation}\label{eq:wpm}
 e^t \underline w(t):=\underline u(r),
 \qquad
 e^t \overline w(t):=\overline u(r).
\end{equation}
Then as $t\to\infty$,
\begin{equation*}
\begin{aligned}
 \underline w(t),\ \overline w(t)&\to 0,\\
 e^t\underline w(t),\ e^t\overline w(t)&\to \infty.
\end{aligned}
\end{equation*}
Define almost everywhere
\begin{equation}\label{eq:flux-definitions}
\begin{aligned}
 P_-(t)&:=|\underline w(t)+\underline w'(t)|^{p-2}
             \bigl(\underline w(t)+\underline w'(t)\bigr),\\
 P_+(t)&:=|\overline w(t)+\overline w'(t)|^{p-2}
             \bigl(\overline w(t)+\overline w'(t)\bigr).
\end{aligned}
\end{equation}
The quantities \(P_\pm\) are locally of bounded variation (BV) and satisfy
the following one-dimen\-sional measure inequalities.

\begin{lemma}
\label{lem:one-dimensional-flux}
$P_\pm(t)$ have locally BV representatives and satisfy
\begin{align}
 -DP_-&\ge \underline w^q\,dt,
 \label{eq:flux-lower}\\
 -DP_+&\le \overline w^q\,dt
 \label{eq:flux-upper}
\end{align}
as Radon-measure inequalities on every sufficiently far right half-line. Here \(DP_\pm\) denotes the distributional derivative of the BV function \(P_\pm\), viewed as a signed Radon measure.
In particular, $P_-$ has a nonincreasing representative.
\end{lemma}

%

\begin{proof}
We have
\begin{equation}\label{eq:exact-flux-transform}
  r^{n-1}|\underline u'|^{p-2}\underline u'
=
-(\frac{n-p}{p-1})^{\,n-1}|\underline w+\underline w'|^{p-2}(\underline w+\underline w')
=
-(\frac{n-p}{p-1})^{\,n-1}P_-(t).
\end{equation}
Fix a nonnegative function
\(
\zeta\in C_c^\infty((T,\infty)),
\)
where \(T\) is large enough so that \(0<r(t)<1\) for \(t>T\), and define the radial test function
\(
\phi(x):=\zeta(t(|x|)).
\)
By \eqref{eq:lower-super},
\begin{equation}\label{eq:radial-test-lower}
|\mathbb S^{n-1}|
\int_0^1
r^{n-1}|\underline u'|^{p-2}\underline u'\phi_r\,dr
\ge
|\mathbb S^{n-1}|
\int_0^1
r^{n-1}\underline u^q\phi\,dr.
\end{equation}
We first transform the left-hand side. By \eqref{eq:exact-flux-transform} and
\[
\phi_r=-\frac{n-p}{(p-1)r}\zeta'(t),
\]
we obtain
\[
r^{n}|\underline u'|^{p-2}\underline u'\phi_r
=
\left(\frac{n-p}{p-1}\right)^n P_-(t)\zeta'(t).
\]
It follows that
\begin{equation}
\int_0^1
r^{n-1}|\underline u'|^{p-2}\underline u'\phi_r\,dr
=
\left(\frac{n-p}{p-1}\right)^{n-1}
\int_{\frac{n-p}{p-1}\log\frac{n-p}{p-1}}^{\infty} P_-(t)\zeta'(t)\,dt.
\label{eq:radial-test-left-transform}
\end{equation}
Next we transform the right-hand side. We have
\[
r^n \underline u(r)^q
=\left(\frac{n-p}{p-1}\right)^n \underline w(t)^q.
\]
Consequently,
\begin{align}\label{eq:radial-test-right-transform}
\int_0^1 r^{n-1}\underline u(r)^q\phi(r)\,dr
=
\left(\frac{n-p}{p-1}\right)^{n-1}
\int_{\frac{n-p}{p-1}\log\frac{n-p}{p-1}}^{\infty} \underline w(t)^q\zeta(t)\,dt.
\end{align}
Substituting \eqref{eq:radial-test-left-transform} and \eqref{eq:radial-test-right-transform} into \eqref{eq:radial-test-lower},
we obtain
\[
\int P_-(t)\zeta'(t)\,dt
\ge
\int \underline w(t)^q\zeta(t)\,dt
\]
for every nonnegative
\(\zeta\in C_c^\infty((T,\infty))\).
Thus
\[
-DP_-\ge \underline w^q\,dt
\]
in the sense of distributions.
The same calculation applied to \eqref{eq:upper-sub} gives
\[
-DP_+\le \overline w^q\,dt
\]
in the sense of distributions.
\(-DP_- -\underline w^q\,dt\) and $DP_+ +\overline w^q\,dt$ are positive distributions, hence positive Radon measures. Since \(\underline w^q\,dt\) and \(\overline w^q\,dt\) are locally finite Radon measures, \(DP_\pm\) are
signed Radon measures. Therefore
\(
P_\pm\in BV_{\rm loc}((T,\infty)).
\)
Moreover,
\[
-DP_-\ge \underline w^q\,dt\ge0,
\]
so \(DP_-\le0\), and hence \(P_-\) has a nonincreasing representative.
\end{proof}


%
%

We next obtain a sharp asymptotic upper bound for the rescaled spherical minimum \(\underline w\).

\begin{lemma}[Sharp upper estimate]\label{lem:sharp-upper}
One has
\begin{equation}\label{eq:lower-envelope-limsup}
 \limsup_{t\to\infty}t^{\frac{1}{q-p+1}}\underline w(t)\le \left(\frac{n-p}{p}\right)^{\frac{n-p}{p(p-1)}}.
\end{equation}
\end{lemma}

\begin{proof}
We first claim that
\[
 P_-(t)>0
\]
for all sufficiently large $t$.
Indeed, if $P_-(t_0)\le0$ at a sufficiently large point, then the monotonicity
from \eqref{eq:flux-lower} gives $P_-(t)\le0$ for $t\ge t_0$.  Since
the map $s\mapsto |s|^{p-2}s$ is strictly increasing and sign-preserving,
$\underline w'+\underline w\le0$ almost everywhere.  Hence
\[
 \frac{d}{dt}(e^t\underline w(t))\le0,
\]
contradicting $e^t\underline w(t)=\underline u(r)\to\infty$.  Consequently
\begin{equation}\label{eq:lower-linear}
 \underline w'(t)+\underline w(t)=P_-(t)^{1/(p-1)}>0
\end{equation}
for almost every sufficiently large $t$.

Fix $L>0$.  Solving \eqref{eq:lower-linear} on $[t-L,t]$ gives
\begin{equation}\label{eq:variation-lower}
 \underline w(t)=e^{-L}\underline w(t-L)
 +\int_{t-L}^t e^{-(t-s)}P_-(s)^{1/(p-1)}\,ds.
\end{equation}
Since $P_-$ is nonincreasing,
\begin{equation}\label{eq:w-lower-A}
 \underline w(t)\ge(1-e^{-L})P_-(t)^{1/(p-1)}.
\end{equation}
Combining \eqref{eq:flux-lower} with \eqref{eq:w-lower-A}, we obtain
\begin{equation}\label{eq:lower-flux-power}
 -DP_-\ge (1-e^{-L})^qP_-^{q/(p-1)}\,dt.
\end{equation}
Let
\(
 \delta:=\frac{q}{p-1}-1.
\)
The chain rule gives
\begin{equation}\label{eq:lower-chain}
 D(P_-^{-\delta})\ge\delta (1-e^{-L})^q\,dt.
\end{equation}
Indeed, fix a compact interval contained in the half-line where $P_->0$.
Write
\[
DP_-=D^aP_-+D^cP_-+D^jP_-,
\]
where $D^aP_-$, $D^cP_-$, and $D^jP_-$ denote respectively the absolutely continuous part, the Cantor part, and the jump part of $DP_-$.
Write
\[
D^aP_-=P'_{-}(t)\,dt.
\]
From \eqref{eq:lower-flux-power},
we have, for the absolutely continuous part,
\[
-P'_{-}(t)\ge (1-e^{-L})^q P_-^{\delta+1}(t)
\quad\text{a.e.}
\]
and by the BV chain rule,
\[
D^a(P_-^{-\delta})
=-\delta P_-^{-(\delta+1)}P'_{-}(t)\,dt
\ge \delta (1-e^{-L})^q\,dt.
\]
Moreover, since $DP_-\le 0$, we have $D^cP_-\le 0$. And at every jump point $t\in \{t:\; P_-(t+)\neq P_-(t-)\}$,
\[
P_-(t+)-P_-(t-)\le 0.
\]
Since $s\mapsto s^{-\delta}$ is decreasing,
\[
D^c(P_-^{-\delta})=-\delta\widetilde P_-^{-\delta-1}\,D^cP_-\ge0,
\]
where $\widetilde P_-$ denotes the precise representative of $P_-$; in one dimension one may take
\[
\widetilde P_-(t):=\frac{P_-(t+)+P_-(t-)}{2}.
\]
Also at every jump point $t_0$,
\[
P_-(t_0+)^{-\delta}-P_-(t_0-)^{-\delta}\ge0.
\]
Thus $D^j(P_-^{-\delta})$ is nonnegative.
Combining the three parts yields
\[
D(P_-^{-\delta})\ge \delta  (1-e^{-L})^q\,dt.
\]
Fix a sufficiently large continuity point $t_0$ of $P_-$.  Integrating
\eqref{eq:lower-chain} gives, for $t>t_0$,
\[
 P_-(t)^{-\delta}
 \ge P_-(t_0)^{-\delta}+\delta  (1-e^{-L})^q(t-t_0).
\]
Thus
\[
 t^{\frac{1}{q-p+1}}P_-(t)^{1/(p-1)}
 \le\left[
 \frac{P_-(t_0)^{-\delta}}{t}
 +\delta  (1-e^{-L})^q\left(1-\frac{t_0}{t}\right)
 \right]^{-\frac{1}{q-p+1}}.
\]
Letting $t\to\infty$ and  $L\to\infty$ yields
\begin{equation}\label{eq:Aminus-limsup}
 \limsup_{t\to\infty}t^{\frac{1}{q-p+1}}P_-(t)^{1/(p-1)}
 \le\left(\frac{p-1}{q-p+1}\right)^{\frac{1}{q-p+1}}
 =\left(\frac{n-p}{p}\right)^{\frac{n-p}{p(p-1)}}.
\end{equation}

Finally, use \eqref{eq:variation-lower} with $L=\sqrt t$.  Since
$\underline w(t)\to0$ as $t\to\infty$, the first term satisfies
\[
 t^{\frac{1}{q-p+1}}e^{-\sqrt t}\underline w(t-\sqrt t)\longrightarrow0.
\]
For the integral term, the monotonicity of $P_-$ gives
\begin{align*}
 &t^{\frac{1}{q-p+1}}
 \int_{t-\sqrt t}^t e^{-(t-s)}P_-(s)^{1/(p-1)}\,ds\le t^{\frac{1}{q-p+1}}(1-e^{-\sqrt t})P_-(t-\sqrt t)^{1/(p-1)}\\
 &\quad=(1-e^{-\sqrt t})
 \left(\frac{t}{t-\sqrt t}\right)^{\frac{1}{q-p+1}}
 (t-\sqrt t)^{\frac{1}{q-p+1}}P_-(t-\sqrt t)^{1/(p-1)}.
\end{align*}
By \eqref{eq:Aminus-limsup}, the last expression has limsup at most
$\left(\frac{n-p}{p}\right)^{\frac{n-p}{p(p-1)}}$.  Combining the two estimates yields \eqref{eq:lower-envelope-limsup}.
\end{proof}

We now obtain the asymptotic lower bound for the rescaled spherical maximum \(\overline w\).

\begin{lemma}[Sharp lower estimate]\label{lem:sharp-lower}
One has
\begin{equation}\label{eq:upper-envelope-liminf}
 \liminf_{t\to\infty}t^{\frac{1}{q-p+1}}\overline w(t)\ge \left(\frac{n-p}{p}\right)^{\frac{n-p}{p(p-1)}}.
\end{equation}
\end{lemma}

\begin{proof}
We have $\overline w>0$, $\overline w\in W^{1,\infty}_{\mathrm{loc}}$, and $\overline w(t)\to0$.
Fix $\varepsilon\in(0,1/2)$.
We will prove that
\begin{equation}\label{eq:good-region}
 -\frac{\overline w'(t)}{\overline w(t)^{q-p+2}}\le \frac{1}{(p-1)(1-\varepsilon)}
 \qquad\text{for almost every sufficiently large }t.
\end{equation}
\eqref{eq:upper-envelope-liminf} will then follow by integrating $(\overline w^{-(q-p+1)})'$; see Step~4.

Choose the right-continuous BV representative of $P_+$, and define
\[
 \Phi(s):=\operatorname{sgn}(s)|s|^{1/(p-1)},
 \qquad
 R(t):=\frac{\overline w(t)-\Phi(P_+(t))}{\overline w(t)^{q-p+2}}.
\]
By \eqref{eq:flux-definitions},
\begin{equation}\label{eq:sharp-lower-R-identity}
 R=-\frac{\overline w'}{\overline w^{q-p+2}}
 \qquad\text{almost everywhere}.
\end{equation}
All pointwise statements about $R$ below refer to this representative.
It is right-continuous and has finite left limits.
Moreover, \eqref{eq:flux-upper} gives $D^s P_+\ge0$, so
$P_+(t)\ge P_+(t-)$.
Since $\Phi$ is increasing and $\overline w$ is continuous,
\begin{equation}\label{eq:sharp-lower-no-upward-jumps}
 R(t)\le R(t-).
\end{equation}
Choose $\delta\in(0,1/2)$ sufficiently small that
\(
 1-(q-p+2)\delta-(1-\varepsilon)\max\{1,(1-\delta)^{2-p}\}>0,
\)
and then choose $T$ sufficiently large that
\begin{equation}\label{eq:w-small}
 \frac{\overline w(t)^{q-p+1}}{(p-1)(1-\varepsilon)}<\delta\qquad(t\ge T).
\end{equation}

\medskip
\noindent\textit{Step 1. On every interval where}
\begin{equation}\label{eq:sharp-lower-working-region}
 R>\frac{1}{(p-1)(1-\varepsilon)},\qquad R\overline w^{q-p+1}<\delta,
\end{equation}
\textit{one has}
\begin{equation}\label{eq:R-decrease}
 DR\le-c_\varepsilon\,dt,
\end{equation}
\textit{where}
\begin{equation*}
 c_\varepsilon:= \frac{1-(q-p+2)\delta-(1-\varepsilon)\max\{1,(1-\delta)^{2-p}\}} {(p-1)(1-\varepsilon)}>0.
\end{equation*}
To prove this, observe that in the range \eqref{eq:sharp-lower-working-region}, we have
\[
 \Phi(P_+)=\overline w(1-R\overline w^{q-p+1})>(1-\delta)\overline w>0.
\]
Consequently, $P_+$ stays away from zero on compact subintervals, and $R\in BV_{\mathrm{loc}}$ there.
Since \(P_+>0\) in this region and the map
\(s\mapsto s^{1/(p-1)}\) is increasing, the BV chain rule, together with
\(D^sP_+\ge0\), gives
\[
D^s\!\left(P_+^{1/(p-1)}\right)\ge0.
\]
Moreover, \(\overline w\in W^{1,\infty}_{\mathrm{loc}}\), so \(\overline w\) and
\(\overline w^{-(q-p+2)}\) have no singular derivative. Hence
\[
D^sR
=
-\overline w^{-(q-p+2)}
D^s\!\left(P_+^{1/(p-1)}\right)
\le0.
\]
For the absolutely continuous part, differentiation and
\eqref{eq:sharp-lower-R-identity} give
\[
 P_+'=-(p-1)\overline w^q(1-R\overline w^{q-p+1})^{p-2}
       \bigl[R'+R-(q-p+2)R^2\overline w^{q-p+1}\bigr].
\]
Since $P_+'\ge-\overline w^q$ almost everywhere by \eqref{eq:flux-upper},
\begin{align*}
 R'
 &\le [-1+(q-p+2)R\overline w^{q-p+1}]R
       +\frac{1}{p-1}(1-R\overline w^{q-p+1})^{2-p}\\
 &\le [-1+(q-p+2)\delta]\frac{1}{(p-1)(1-\varepsilon)}+\frac{\max\{1,(1-\delta)^{2-p}\}}{p-1}
 =-c_\varepsilon.
\end{align*}
Together with $D^sR\le0$, this proves \eqref{eq:R-decrease}.
In particular, while \eqref{eq:sharp-lower-working-region} holds, both $R$ and $\overline w^{q-p+1}$ are positive and nonincreasing.

\medskip
\noindent\textit{Step 2.  $R\le 1/[{(p-1)(1-\varepsilon)}]$ is reached in finite time.}

We first use nonremovability to find a starting point.
For every $\kappa>0$, it is impossible that
\begin{equation}\label{eq:eventual-exponential}
 \overline w'(t)\le-\kappa \overline w(t)
 \qquad\text{for almost every sufficiently large }t.
\end{equation}
Indeed, this would imply $\overline w(t)\le Ce^{-\kappa t}$.
If $0<\kappa<1$, the change of variables \eqref{eq:EF}--\eqref{eq:wpm}
would give
\[
 u(x)\le\overline u(|x|)
 \le C|x|^{-\frac{n-p}{p-1}(1-\kappa)}.
\]
Hence, $u^{q(1+\theta)}$ would be locally integrable for some $\theta>0$,
and \cite[Theorem~10]{Serrin1964} would make the singularity removable.
If $\kappa\ge1$, $u$ would be bounded near zero and hence removable by
Lemma~\ref{lem:dichotomy}.
Both conclusions contradict the standing nonremovability assumption.

Applying \eqref{eq:eventual-exponential} and \eqref{eq:sharp-lower-R-identity} with $\kappa=\delta$, we may choose $t_0>T$
such that
\[
 R(t_0)\overline w(t_0)^{q-p+1}<\delta.
\]
If $R(t_0)\le 1/[{(p-1)(1-\varepsilon)}]$, there is nothing to prove in this step.
Suppose instead that $R(t_0)>1/[{(p-1)(1-\varepsilon)}]$, and let $[t_0,\tau)$ be the maximal
right interval on which \eqref{eq:sharp-lower-working-region} holds.
By Step~1, for $t_0\le t<\tau$,
\begin{equation}\label{eq:sharp-lower-product-control}
 R(t)\overline w(t)^{q-p+1}\le R(t_0)\overline w(t_0)^{q-p+1}<\delta,
\end{equation}
and
\[
 R(t)\le R(t_0)-c_\varepsilon(t-t_0).
\]
The latter estimate cannot hold together with $R>\frac{1}{(p-1)(1-\varepsilon)}$ on an infinite interval,
so $\tau<\infty$.
By \eqref{eq:sharp-lower-no-upward-jumps}, continuity of $\overline w$, and
\eqref{eq:sharp-lower-product-control},
\[
 R(\tau)\overline w(\tau)^{q-p+1}
 \le R(\tau-)\overline w(\tau)^{q-p+1}
 \le R(t_0)\overline w(t_0)^{q-p+1}<\delta.
\]
If $R(\tau)>1/[{(p-1)(1-\varepsilon)}]$, right
continuity would extend the interval beyond $\tau$.
This contradicts maximality. Therefore $R(\tau)\le 1/[{(p-1)(1-\varepsilon)}]$.
We have thus found $s>T$ with $R(s)\le 1/[{(p-1)(1-\varepsilon)}]$.

\medskip
\noindent\textit{Step 3. Once $R\le 1/[{(p-1)(1-\varepsilon)}]$ holds, it continues to hold for all later times.}

The set
\[
 A:=\{t>T:R(t)>\frac{1}{(p-1)(1-\varepsilon)}\}
\]
is open.
Indeed, if \(t_0\) belongs to the set, then right continuity preserves the strict inequality for \(t>t_0\) sufficiently close to \(t_0\).
Moreover, since $R(t-)\ge R(t)>1/[{(p-1)(1-\varepsilon)}]$, the existence of the left limit preserves the strict inequality for \(t<t_0\) sufficiently close to \(t_0\).

Suppose that $R(t)> 1/[{(p-1)(1-\varepsilon)}]$ at some time after $s$, and choose a connected component \((a,b)\) contained in \(A\cap(s,\infty)\).
Then
\[
 R(a)=\frac{1}{(p-1)(1-\varepsilon)}.
\]
Indeed, if $a=s$, then $R(s)=R(a)\le 1/[{(p-1)(1-\varepsilon)}]$; if $a>s$, then $R(a)\le 1/[{(p-1)(1-\varepsilon)}]$ follows from maximality of the component; $R(a)\ge 1/[{(p-1)(1-\varepsilon)}]$ follows from right continuity and
$R>1/[{(p-1)(1-\varepsilon)}]$ on $(a,b)$.
By \eqref{eq:w-small},
\[
R(a)\overline w(a)^{q-p+1}=\frac{\overline w(a)^{q-p+1}}{(p-1)(1-\varepsilon)}<\delta.
\]
Right continuity therefore gives $h>0$, with $a+h<b$, such that
\eqref{eq:sharp-lower-working-region} holds on $(a,a+h)$.
Integrating \eqref{eq:R-decrease} from $a'>a$ to $t\in(a',a+h)$ and
then letting $a'\downarrow a$, we obtain
\[
 R(t)\le \frac{1}{(p-1)(1-\varepsilon)}-c_\varepsilon(t-a)<\frac{1}{(p-1)(1-\varepsilon)},
\]
contrary to $(a,a+h)\subset A$.
Thus $R(t)\le 1/[{(p-1)(1-\varepsilon)}]$ for all $t\ge s$.

 \eqref{eq:good-region} is the direct consequence of Step 3.

\medskip
\noindent\textit{Step 4. Integrate to obtain \eqref{eq:upper-envelope-liminf}.}
By \eqref{eq:sharp-lower-R-identity},
\[
 (\overline w^{-(q-p+1)})'=(q-p+1)R\le \frac{q-p+1}{(p-1)(1-\varepsilon)}
 \qquad\text{almost everywhere on }(s,\infty).
\]
Since $\overline w^{-(q-p+1)}$ is locally absolutely continuous,
\[
 \overline w(t)^{-(q-p+1)}\le \overline w(s)^{-(q-p+1)}+\frac{q-p+1}{(p-1)(1-\varepsilon)}(t-s).
\]
It follows that
\[
 \liminf_{t\to\infty}t^{\frac{1}{q-p+1}}\overline w(t)
 \ge\left(\frac{q-p+1}{(p-1)(1-\varepsilon)}\right)^{-\frac{1}{q-p+1}}
 =\left(\frac{(p-1)(1-\varepsilon)}{q-p+1}\right)^{\frac{1}{q-p+1}}.
\]
Letting $\varepsilon\downarrow0$ proves \eqref{eq:upper-envelope-liminf}.
\end{proof}


\begin{proof}[Proof of Theorem~\ref{thm:main}]
If the singularity is removable, alternative~\textup{(i)} has already
been established in Lemma~\ref{lem:dichotomy}.  Assume that it is
nonremovable.  Lemmas~\ref{lem:sharp-upper} and~\ref{lem:sharp-lower} give
\[
 \limsup_{t\to\infty}t^{\frac{1}{q-p+1}}\underline w(t)\le \left(\frac{n-p}{p}\right)^{\frac{n-p}{p(p-1)}}
 \quad \text{and} \quad
 \liminf_{t\to\infty}t^{\frac{1}{q-p+1}}\overline w(t)\ge \left(\frac{n-p}{p}\right)^{\frac{n-p}{p(p-1)}}.
\]
By Lemma~\ref{lem:asymptotic-symmetry},
\[
 \frac{\overline w(t)}{\underline w(t)}
 =\frac{\overline u(r)}{\underline u(r)}\longrightarrow1.
\]
Therefore
\begin{equation}\label{eq:w-exact}
 \lim_{t\to\infty}t^{\frac{1}{q-p+1}}\underline w(t)
 =
 \lim_{t\to\infty}t^{\frac{1}{q-p+1}}\overline w(t)
 =\left(\frac{n-p}{p}\right)^{\frac{n-p}{p(p-1)}}.
\end{equation}
For $|x|=r$,
\[
 e^t\underline w(t)\le u(x)\le e^t\overline w(t).
\]
Since
\[
 e^t=\left(\frac{n-p}{(p-1)r}\right)^{\frac{n-p}{p-1}},
 \qquad
  t=\frac{n-p}{p-1}\log\frac{n-p}{(p-1)r},
\]
we have
\[
 \frac{\log(1/r)}{t}
 \to \frac{p-1}{n-p}
 \quad \text{as } r\downarrow0.
\]
For the lower spherical minimum, \eqref{eq:w-exact} gives
\begin{align*}
  r^{\frac{n-p}{p-1}}
  \left(\log\frac1r\right)^{\frac{n-p}{p(p-1)}}e^t\underline w(t)
  &=\left(\frac{n-p}{p-1}\right)^{\frac{n-p}{p-1}}
    \left(\frac{\log(1/r)}{t}\right)^{\frac{n-p}{p(p-1)}}
    \bigl[t^{\frac{1}{q-p+1}}\underline w(t)\bigr]\\
  &\longrightarrow
  \left(\frac{n-p}{p-1}\right) ^{\frac{n-p}{p}} \left(\frac{n-p}{p}\right)^{\frac{n-p}{p(p-1)}}  =A_{n,p}.
\end{align*}
The same limit holds with $\underline w$ replaced by $\overline w$.
The squeeze theorem yields \eqref{eq:final-limit}. This completes the proof of Theorem~\ref{thm:main}.
\end{proof}

%

\section{Proof of Theorem~\ref{prop:radial-existence}}
\label{sec:radial-existence}

\begin{proof}
We use the change of variables \eqref{eq:EF}, written as
\[
 t=\frac{n-p}{p-1}\log\frac{n-p}{(p-1)r},
 \qquad r=\frac{n-p}{p-1}e^{-\frac{p-1}{n-p}t},
 \qquad U(r)=e^t w(t).
\]
Thus $0<r\le1$ corresponds to $t \ge  t_0:= \frac{n-p}{p-1} \log\frac{n-p}{p-1}$.
%
By assumption, $U$ solves the radial equation
\begin{equation}\label{eq:43}
 -\bigl(r^{n-1}|U'|^{p-2}U'\bigr)'=r^{n-1}U^q.
\end{equation}
We restrict our
construction to profiles satisfying $w+w'>0$, which is
equivalent to $U'(r)<0$ under the above change of variables.
Set
\begin{equation}\label{eq:44}
 P:=(w+w')^{p-1}.
\end{equation}
Since $\frac{n-p}{p-1}q=n$, direct calculation gives
\begin{align*}
 U'(r)&=-\frac{n-p}{(p-1)r}e^t P^{\frac1{p-1}},\\
 r^{n-1}|U'|^{p-2}U'&=-\left(\frac{n-p}{p-1}\right)^{n-1}P,\\
 r^nU^q&=\left(\frac{n-p}{p-1}\right)^nw^q.
\end{align*}
Thus, \eqref{eq:43} and \eqref{eq:44} are equivalent to
\begin{equation}\label{eq:existence-system}
 \begin{cases}
  w'=P^{\frac1{p-1}}-w,\\
  P'=-w^q.
 \end{cases}
\end{equation}
Conversely, every positive solution $(w,P)$ of
\eqref{eq:existence-system} on $[t_0,\infty)$ determines a positive,
strictly decreasing radial solution of \eqref{eq:main}.
Introduce
\[
 Z:=\frac{w}{P^{\frac1{p-1}}}.
\]
We obtain the equivalent system
\begin{equation}\label{eq:existence-ratio-system}
 \begin{cases}
  P'=-Z^qP^{\frac{n}{n-p}},\\
  Z'=1-Z+\frac1{p-1}Z^{q+1}P^{\frac{p}{n-p}}.
 \end{cases}
\end{equation}
%
Choose $\varepsilon=\varepsilon(n,p)>0$ sufficiently small that
\[
 \frac{2^{q+1}}{p-1}\varepsilon^{\frac{p}{n-p}}\le\frac12,
\]
and consider
\[
 \mathcal D:=\{(P,Z):0<P\le\varepsilon,\ 1\le Z\le2\}.
\]
On the boundary $Z=1$, the second equation in
\eqref{eq:existence-ratio-system} gives
\[
 Z'=\frac1{p-1}P^{\frac{p}{n-p}}>0,
\]
whereas on $Z=2$ it gives
\[
 Z'=-1+\frac{2^{q+1}}{p-1}P^{\frac{p}{n-p}}\le-\frac12<0.
\]
Moreover, $P'<0$ throughout $\mathcal D$. Hence, a trajectory cannot
leave $\mathcal D$ through $Z=1$, $Z=2$, or $P=\varepsilon$; see Figure~\ref{fig:existence-invariant-region}.

\begin{figure}[htbp]
 \centering
 \includegraphics[width=0.55\textwidth]{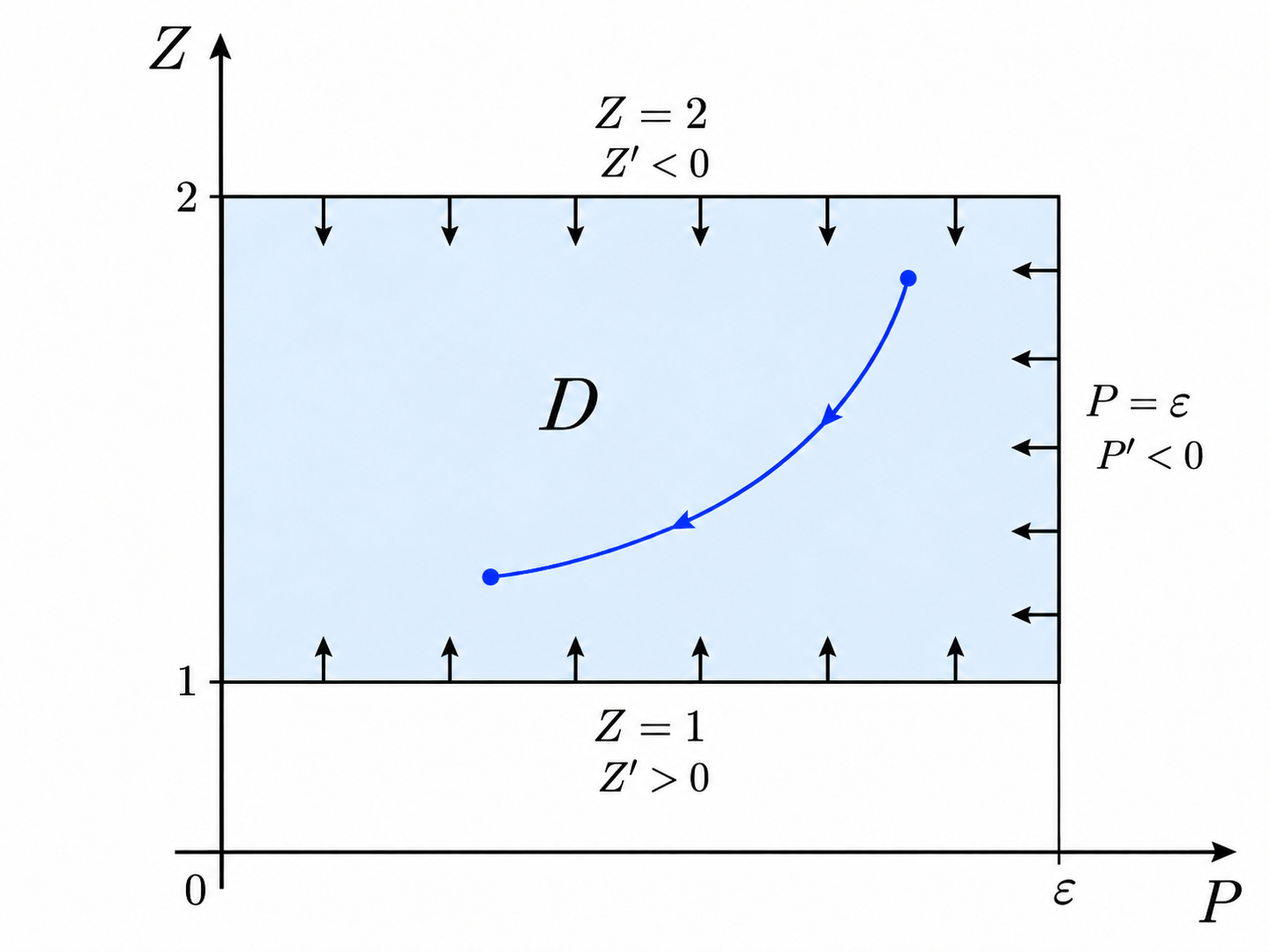}
 \caption{The region $\mathcal D$ in the $(P,Z)$-plane. The blue curve is a schematic trajectory.}
 \label{fig:existence-invariant-region}
\end{figure}

Fix $P_0\in(0,\varepsilon)$ and impose the initial conditions
\[
 P(t_0)=P_0,
 \qquad Z(t_0)=\frac32.
\]
The vector field in \eqref{eq:existence-ratio-system} is smooth in the
positive quadrant, so these data determine a unique maximal solution.
As long as this solution remains in the positive quadrant, the preceding
boundary inequalities imply
\begin{equation}\label{eq:existence-trapping}
 0<P(t)\le P_0,
 \qquad 1\le Z(t)\le2.
\end{equation}
Furthermore,
\begin{equation}\label{eq:49}
   \frac{d}{dt}P^{-\frac{p}{n-p}}=\frac{p}{n-p}Z^q\in (\frac{p}{n-p},\frac{p2^q}{n-p}),
\end{equation}
so integration yields
\begin{equation}\label{eq:existence-bounds}
 P_0^{-\frac{p}{n-p}}+\frac{p}{n-p}(t-t_0)
 \le P(t)^{-\frac{p}{n-p}}
 \le P_0^{-\frac{p}{n-p}}+\frac{p\,2^q}{n-p}(t-t_0).
\end{equation}
In particular,
\begin{equation}\label{50}
   P(t)\ge
 \left[P_0^{-\frac{p}{n-p}}+\frac{p\,2^q}{n-p}(t-t_0)\right]^{-\frac{n-p}{p}}>0
\end{equation}
for every finite $t$ in the interval of existence. If the maximal forward
existence time $T_{\max}$ of the solution $(P,Z)$ to \eqref{eq:existence-ratio-system}
were finite, \eqref{50} and
\eqref{eq:existence-trapping} would keep $(P(t),Z(t))$, for every
$t\in[t_0,T_{\max})$, in the compact range
\[
 \left[
  \left(P_0^{-\frac{p}{n-p}}
   +\frac{p\,2^q}{n-p}(T_{\max}-t_0)\right)^{-\frac{n-p}{p}},
  P_0
 \right]\times[1,2].
\]
The standard continuation theorem for
ordinary differential equations would then extend the solution beyond
that time, a contradiction. Thus, the solution exists on
$[t_0,\infty)$, and $(P(t),Z(t))\in \mathcal D$ for all $t\ge t_0$.
The left-hand inequality in \eqref{eq:existence-bounds} also gives
\[
 P(t) \to 0
 \quad \text{as } t\to\infty.
\]
We recover a positive solution $w(t)=Z(t)P^\frac{1}{p-1}(t)$ of
\eqref{eq:existence-system}, with $w(t)\to0$.

Write the second equation in \eqref{eq:existence-ratio-system} as
\[
 Z'+Z=1+F(t),
 \qquad
 F(t):=\frac1{p-1}Z(t)^{q+1}P(t)^{\frac{p}{n-p}}.
\]
We have $F(t)\to0$.
For $t\ge T\ge t_0$, the variation-of-constants formula gives
\[
 Z(t)-1
 =e^{-(t-T)}(Z(T)-1)
  +\int_T^t e^{-(t-s)}F(s)\,ds.
\]
Consequently,
\[
 0\le\limsup_{t\to\infty}(Z(t)-1)
 \le\sup_{s\ge T}F(s).
\]
Letting $T\to\infty$, we conclude that $Z(t)\to1$.
By \eqref{eq:49} and l'H\^{o}pital's rule, we find
\[
 \frac{P(t)^{-\frac{p}{n-p}}}{t}\to \frac{p}{n-p}\qquad \text{as } t\to\infty.
\]
It follows that
\begin{equation}\label{eq:existence-profile-asymptotics}
 P(t)\sim\left(\frac{p}{n-p}t\right)^{-\frac{n-p}{p}},
 \quad
 w(t)=Z(t)P(t)^{\frac1{p-1}}
 \sim\left(\frac{p}{n-p}t\right)^{-\frac{n-p}{p(p-1)}}
 \quad \text{as } t\to\infty.
\end{equation}
%
The vector field in \eqref{eq:existence-system} is smooth when $w,P>0$, so $u(x):=U(|x|)$ belongs to $C^\infty(B_1\setminus\{0\})$.
We have
\[
 r^{\frac{n-p}{p-1}}e^{t(r)}
 =\left(\frac{n-p}{p-1}\right)^{\frac{n-p}{p-1}},
 \quad
 \frac{\log(1/r)}{t(r)}\to \frac{p-1}{n-p}
 \quad \text{as } r\downarrow0.
\]
Moreover, since $t(r)\to\infty$ as $r\downarrow0$,
\eqref{eq:existence-profile-asymptotics} gives
\[
 t(r)^{\frac{n-p}{p(p-1)}}w(t(r))
 \to
 \left(\frac{p}{n-p}\right)^{-\frac{n-p}{p(p-1)}}.
\]
Thus,
\begin{align*}
 \lim_{r\downarrow0}r^{\frac{n-p}{p-1}}
 \left(\log\frac1r\right)^{\frac{n-p}{p(p-1)}}U(r)
 &=\lim_{r\downarrow0}\left[
    r^{\frac{n-p}{p-1}}e^{t(r)}
    \left(\frac{\log(1/r)}{t(r)}\right)^{\frac{n-p}{p(p-1)}}
    t(r)^{\frac{n-p}{p(p-1)}}w(t(r))
   \right]\\
 &=\left(\frac{n-p}{p-1}\right)^{\frac{n-p}{p-1}}
   \left(\frac{p-1}{n-p}\right)^{\frac{n-p}{p(p-1)}}
   \left(\frac{p}{n-p}\right)^{-\frac{n-p}{p(p-1)}}\\
 &=A_{n,p}.
\end{align*}
In particular, $U(r)\to\infty$ as $r\downarrow0$. This completes the proof of Theorem~\ref{prop:radial-existence}.
\end{proof}



\medskip
\noindent\textbf{Acknowledgements} Part of this work was completed while H. Yang was visiting The Chinese University of Hong Kong. He would like to thank Professor Juncheng Wei for his support and encouragement, as well as the Department of Mathematics and the Institute of Mathematical Sciences for their hospitality. The research of M. Xu is partially supported by the National Key R\&D
Program of China 2025YFA1017600 and NSFC 12526202. The research of H. Yang is supported by NSFC 12301140 and the Shanghai Frontier Science Center of Modern Analysis.

\bigskip
\begingroup
\hypersetup{urlcolor=black}
\noindent\begin{minipage}{\textwidth}
M.~Xu\\
School of Mathematical Sciences, Fudan University\\
Shanghai 200433, China\\
Email: \href{mailto:meiqing_xu@fudan.edu.cn}{\textsf{meiqing\_xu@fudan.edu.cn}}
\end{minipage}

\bigskip

\noindent\begin{minipage}{\textwidth}
H.~Yang\\
School of Mathematical Sciences, Shanghai Jiao Tong University\\
Shanghai 200240, China\\
Email: \href{mailto:hui-yang@sjtu.edu.cn}{\textsf{hui-yang@sjtu.edu.cn}}
\end{minipage}
\endgroup

\end{document}